\documentclass{article}
\usepackage[usenames, dvipsnames]{xcolor}
\usepackage[hyphens]{xurl}
\usepackage[colorlinks=true, breaklinks=true, linkcolor=Blue, urlcolor=Blue, citecolor=Black]{hyperref}
\newcommand{\defin}[1]{\textcolor{Maroon}{\emph{#1}}\index{#1}}
\usepackage[english]{babel}
\addto\extrasenglish{}
\addto\extrasenglish{}
\usepackage{mathtools}
\usepackage{sectsty}
\usepackage{float}
\usepackage{graphicx}
\usepackage{tocloft}
\newcommand{\sm}{\smallsetminus}
\usepackage{makeidx}

\sectionfont{\fontsize{15}{17}\selectfont}
\sectionfont{\fontsize{15}{17}\selectfont}

\usepackage{amsmath}

\usepackage[a4paper, margin=2.5cm]{geometry}
\usepackage{csquotes}
\usepackage{enumitem}
\MakeOuterQuote{"}
\usepackage{tikz}
\usetikzlibrary{arrows.meta,decorations.markings}
\usepackage{subcaption}
\usepackage[backend=biber,style=numeric, 
maxbibnames=20, 
url=false, 
isbn=false, 
doi=false]{biblatex}
\DeclareLabelalphaTemplate{
  \labelelement{
    \field[final]{shorthand}
    \field{label}
    \field[strwidth=3,strside=left,ifnames=1]{labelname}
    \field[strwidth=1,strside=left]{labelname}
  }
  \labelelement{
    \field{year}    
  }
}

\newbibmacro{string+doiurlisbn}[1]{%
  \iffieldundef{doi}{%
    \iffieldundef{url}{%
      \iffieldundef{isbn}{%
        \iffieldundef{issn}{%
          #1%
        }{%
          \href{http://books.google.com/books?vid=ISSN\thefield{issn}}{#1}%
        }%
      }{%
        \href{http://books.google.com/books?vid=ISBN\thefield{isbn}}{#1}%
      }%
    }{%
      \href{\thefield{url}}{#1}%
    }%
  }{%
    \href{http://dx.doi.org/\thefield{doi}}{#1}%
  }%
}

\DeclareFieldFormat{title}{%
  \usebibmacro{string+doiurlisbn}{\mkbibemph{#1}}}
\DeclareFieldFormat
  [article,inbook,incollection,inproceedings,patent,thesis,unpublished]
  {title}{\usebibmacro{string+doiurlisbn}{\mkbibquote{#1\isdot}}}
\DeclareFieldFormat
  [suppbook,suppcollection,suppperiodical]
  {title}{\usebibmacro{string+doiurlisbn}{#1}}

\usepackage{authblk}

\newcommand{\LLL}{{\mathcal L}_{\mathrm{out}}}
\newcommand{\SSS}{{\mathcal S}_{\mathrm{out}}}
\newcommand{\KKK}{{\mathcal K}_{\mathrm{out}}}

\usepackage{amsmath,amssymb,amsthm}
\usepackage{thmtools}
\usepackage{thm-restate}

\declaretheorem[name=Lemma, numberwithin = section]{lemma}
\declaretheorem[name=Theorem,sibling = lemma]{theorem}
\declaretheorem[name=Proposition, sibling=lemma]{proposition}

\declaretheorem[name=Definition, sibling=lemma]{definition}
\declaretheorem[name=Corollary, sibling=lemma]{corollary}

\declaretheorem[name=Notation]{notation}
\declaretheorem[name=Example]{example}
\declaretheorem[name=Remark, sibling=lemma]{remark}

\newcommand{\spec}{\mathsf{spec}}

\usepackage[noabbrev,capitalise,nameinlink]{cleveref}
\crefname{theorem}{Theorem}{Theorems}
\crefname{proposition}{Proposition}{Propositions}
\crefname{lemma}{Lemma}{Lemmas}
\crefname{claim}{Claim}{Claims}
\crefname{subclaim}{Sub-Claim}{Sub-Claims}
\crefname{observation}{Observation}{Observations}
\crefname{remark}{Remark}{Remarks}
\crefname{corollary}{Corollary}{Corollaries}
\crefname{definition}{Definition}{Definitions}
\crefname{conjecture}{Conjecture}{Conjectures}
\crefname{question}{Question}{Questions}
\crefformat{equation}{(#2#1#3)}
\Crefformat{equation}{Equation #2(#1)#3}

\usepackage{chngcntr}

\usepackage[classfont=sanserif, langfont=roman,funcfont=italic]{complexity}

\title{\LARGE{On the Spectra of Digraph Laplacians}}

\author{
S. Akbari\thanks{Department of Mathematical Sciences, Sharif University of Technology, Tehran,
Iran. Email:  s\_akbari@sharif.edu},  
M. Cavers\thanks{Department of Computer and Mathematical Sciences, University of Toronto Scarborough, Toronto, Ontario, M1C 1A4, Canada. Email: michael.cavers@utoronto.ca}, 
F. Maghsoudi\thanks{Email: farzad.maghsoudi93@gmail.com}, 
B. Miraftab\thanks{School of Computer Science,
Carleton University, Ottawa, Ontario, K1S 5B6, Canada. Email: bobby.miraftab@gmail.com}
}

\date{\today}

\begin{document}

\maketitle

\fontsize{12}{16}\selectfont

\begin{abstract}
We present several Laplacian-type matrices associated with a loopless digraph $D$: the out-/in-degree Laplacians $\mathcal L_{\mathrm{out}},\mathcal L_{\mathrm{in}}$, the incidence Laplacian $\mathcal L_{\mathrm{inc}}=BB^{\mathsf T}$, and the symmetrized and skew-symmetrized variants $\mathcal S_{\mathrm{out}},\mathcal K_{\mathrm{out}}$.
We show that $\mathcal L_{\mathrm{inc}}(D)$ coincides with the Laplacian of the underlying undirected multigraph, and we derive spectral and characteristic-polynomial relations under arc reversal and complementation (including a simplification for Eulerian digraphs for $\mathcal S_{\mathrm{out}}$).
We demonstrate that the spectral radius of $\mathcal L_{\mathrm{out}}$ is bounded above by the order of the digraph and give a characterization in the equality case.
We further obtain explicit formulas for joins and line digraphs, giving a general determinantal identity relating the out-degree Laplacian characteristic polynomials of a regular digraph and its line digraph.
\end{abstract}





\section{Introduction}
We use standard digraph terminology, following \cite{brualdi2010spectra}. Unless stated otherwise, $D$ is a finite loopless digraph with vertex set $V(D)=\{v_1,\ldots,v_n\}$ and arc set $E$ containing $m$ arcs; multiple arcs are allowed. 
We write $A=A(D)$ for the adjacency matrix of $D$. 
Thus $A_{ij}$ is the number of arcs from $v_i$ to $v_j$, and $A_{ii}=0$ for all $i$. 
Complements (see \cref{def:comp}) and tournaments (i.e., an orientation of a complete graph) are always assumed to be simple (i.e., no loops or multiple arcs, but $2$-cycles are permitted).
The \defin{outdegree matrix} and \defin{indegree matrix} of $D$ are
$D_{\mathrm{out}}(D)=\mathrm{diag}(d_1^{+},\dots,d_n^{+})$
and $D_{\mathrm{in}}(D)=\mathrm{diag}(d_1^{-},\dots,d_n^{-})$, respectively,
where $d_i^{+}\coloneqq \sum_{j} A_{ij}$ and $d_i^{-}\coloneqq \sum_{j} A_{ji}$ denote the out-degree and in-degree of vertex $v_i$.
We also use the notation $d_{v_i}^+=d_i^+$ and $d_{v_i}^-=d_i^-$ for $v_i\in V$.
A digraph $D$ is \defin{Eulerian} if $D_{\mathrm{in}}(D)=D_{\mathrm{out}}(D)$ and \defin{$r$-out-regular} if $D_{\mathrm{out}}(D)=rI$.
The \defin{oriented incidence matrix} $B(D)$ of $D$ is the $n\times m$ matrix whose rows are indexed by vertices and whose columns are indexed by arcs. If $e_j\in E$ has tail $v_p$ and head $v_q$, then
\[
B_{ij}=
\begin{cases}
-1, & \text{if } i=p,\\
\phantom{-}1, & \text{if } i=q,\\
\phantom{-}0, & \text{otherwise.}
\end{cases}
\]
For a square matrix $M$ we write
$\phi_M(x)\coloneqq\det(xI-M)$ for its characteristic polynomial and $\spec(M)$ for its multiset of eigenvalues. 

Spectral graph theory traditionally studies undirected graphs via the adjacency matrix and the (combinatorial) Laplacian.
For digraphs, several nonequivalent Laplacian-type matrices have been proposed in the literature, e.g., see \cite{boley2011commute,brualdi1991combinatorial,MR2223513,veerman2020primer}, reflecting different ways in which
orientation interacts with degree and incidence.
Compared to the undirected setting, spectral theory for digraphs is less standardized.
Nevertheless, directed spectral methods have attracted increasing attention in recent years, see \cite{cavers2025spectra,cavers2025digraphs,cavers2026spectra}.
The purpose of this paper is to initiate a systematic study of
Laplacian spectra for digraphs, with an emphasis on the \defin{out-degree Laplacian}, denoted by $\mathcal L_{\mathrm{out}}$ and defined in \cref{def:lap}, and on the relationships among
its standard variants.
It is worth mentioning that it is known that a matrix-tree theorem holds for $\mathcal L_{\mathrm{out}}$ (see \cite[Theorem~9.6.1]{brualdi1991combinatorial}) and that $\mathcal L_{\mathrm{out}}$ is sometimes called the unnormalized Laplacian and has been used in several applications; see, for example, \cite{boley2011commute} and the references therein.

\begin{definition}\label{def:lap}
Let $D$ be a digraph and $B(D)$ be its oriented incidence matrix.
\begin{itemize}
\item The \defin{incidence Laplacian} is
$\mathcal L_{\mathrm{inc}}(D)\coloneqq B(D)B(D)^{\mathsf T}$.
\item The \defin{out-degree Laplacian} and \defin{in-degree Laplacian} are
\[
\mathcal L_{\mathrm{out}}(D)\coloneqq D_{\mathrm{out}}(D)-A(D),\qquad
\mathcal L_{\mathrm{in}}(D)\coloneqq D_{\mathrm{in}}(D)-A(D).
\]

\item The \defin{symmetrized out-degree Laplacian} and
\defin{symmetrized in-degree Laplacian} are
\[
\mathcal S_{\mathrm{out}}(D)
\coloneqq \tfrac12\bigl(\mathcal L_{\mathrm{out}}(D)+\mathcal L_{\mathrm{out}}(D)^{\mathsf T}\bigr),
\qquad
\mathcal S_{\mathrm{in}}(D)
\coloneqq \tfrac12\bigl(\mathcal L_{\mathrm{in}}(D)+\mathcal L_{\mathrm{in}}(D)^{\mathsf T}\bigr).
\]
\item The \defin{skew-symmetrized out-degree Laplacian} and \defin{skew-symmetrized in-degree Laplacian} are
\[
\mathcal K_{\mathrm{out}}(D)\coloneqq \tfrac12\bigl(\LLL(D)-\LLL(D)^{\mathsf T}\bigr),
\qquad 
\mathcal K_{\mathrm{in}}(D)\coloneqq \tfrac12\bigl(\mathcal L_{\mathrm{in}}(D)-\mathcal L_{\mathrm{in}}(D)^{\mathsf T}\bigr).
\]
\end{itemize}
\end{definition}
This paper studies the Laplacian-type matrices defined in \cref{def:lap}.  
One motivation for considering the symmetrized and skew-symmetrized out-degree Laplacians come from \cite{boley2011commute,li2012digraph}.
In \cite{li2012digraph}, the authors introduce a normalized digraph
Laplacian, denoted by $\Gamma$, and use $\nabla=(\Gamma-\Gamma^{\mathsf T})/2$ to quantify the asymmetry of a digraph,
whereas in \cite{boley2011commute}, the authors study the symmetrized Laplacians $L^s=(L+L^{\mathsf T})/2$ and $L^{ds}=(L^d+(L^d)^{\mathsf T})/2$, where $L$ and $L^d$ denote the ordinary and diagonally scaled Laplacian, respectively.
Our definition for $\mathcal K_{\mathrm{out}}$ and $\mathcal S_{\mathrm{out}}$ are in the same spirit, but are formed from the unnormalized out-degree Laplacian instead.


Unlike the undirected setting, different ``directed Laplacians'' capture genuinely different
information.  For example, $\SSS(D)$ is always symmetric and hence has real spectrum, while $\LLL(D)$ may have complex
eigenvalues; moreover, the incidence Laplacian coincides with the Laplacian of the underlying undirected multigraph
obtained from $D$ by forgetting orientations (\Cref{lem:forget}).  Thus, to detect directed phenomena one must look beyond $\mathcal L_{\mathrm{inc}}$.

This paper develops several basic structural operations and their spectral consequences.
We first analyze reversal of all arcs and record the induced transformations of
$\mathcal L_{\mathrm{out}}$, $\mathcal L_{\mathrm{in}}$, $\SSS$, $\mathcal S_{\mathrm{in}}$, $\KKK$, and $\mathcal L_{\mathrm{inc}}$
(\Cref{sec:reverse}, see~\Cref{lem:reverse-all}).
Next, we study complements of simple digraphs and derive characteristic polynomial identities for
$\mathcal L_{\mathrm{inc}}$, $\SSS$, and $\KKK$ (\Cref{sec:complement},
~\Cref{lem:comp-poly-variants}), extending classical complement relations for graphs.
We then turn to line digraphs: although the adjacency characteristic polynomial behaves cleanly under the line digraph
construction, the out-degree Laplacian does not; we give an exact determinantal relation between the characteristic polynomial of
$\mathcal L_{\mathrm{out}}$ of a digraph $D$ and that of its line digraph (\Cref{thm:outlap-line-general}) and develop corresponding reductions for $\SSS$, $\KKK$, and $\mathcal L_{\mathrm{inc}}$.
\Cref{sec:join} treats joins of digraphs, yielding explicit spectral decompositions and characteristic-polynomial
factorizations, while \Cref{sec:bipartite} discusses bipartite digraphs.

This work is intended as a foundation for a broader theory of Laplacian spectra of digraphs.
In a companion paper we will pursue extremal questions, spectral radius bounds, and applications to special classes of
digraphs.

\section{Preliminaries}
The following lemma is immediate from the definitions.

\begin{lemma}\label{lem:prim_lap}
Let $D$ be a digraph with adjacency matrix $A$. Then the following hold:
\begin{enumerate}
\item  $\LLL(D)=\SSS(D)+\mathcal K_{\mathrm{out}}(D)$, and $\mathcal L_{\mathrm{in}}(D)=\mathcal S_{\mathrm{in}}(D)+\mathcal K_{\mathrm{in}}(D)$, that is, these are the decompositions into symmetric ($\SSS(D)$ and $\mathcal S_{\mathrm{in}}(D)$) and skew-symmetric ($\KKK(D)$ and $\mathcal K_{\mathrm{in}}(D)$) parts.
\item  $\SSS(D)=D_{\mathrm{out}}(D)-\tfrac12\big(A+A^{\mathsf T}\big)$,  and $S_{\mathrm{in}}(D)=D_{\mathrm{in}}(D)-\tfrac12\big(A+A^{\mathsf T}\big)$.
\item  $\KKK(D)=\mathcal K_{\mathrm{in}}(D)=\tfrac12(A^{\mathsf T}-A)$.
\item All eigenvalues of $\SSS(D)$ and $\mathcal{S}_{\mathrm{in}}(D)$ are real, while all eigenvalues of $\KKK(D)$ and $\mathcal K_{\mathrm{in}}(D)$ are purely imaginary.
\end{enumerate}
\end{lemma}

Since \cref{lem:prim_lap} implies $\mathcal K_{\mathrm{in}}(D)=\KKK(D)$, we do not consider $\mathcal K_{\mathrm{in}}(D)$ further in this paper and simply refer to $\KKK(D)$ as the skew-symmetrized Laplacian.

The following lemma shows that the incidence Laplacian does not capture genuinely directed information: it is precisely the Laplacian of the underlying undirected multigraph obtained by forgetting the orientations of the arcs.
For this reason, we focus on the out-degree Laplacian in this paper.
Although the lemma is a known result, we could not find any proper reference, thus provide a short proof here for the convenience of the reader.

\begin{lemma}\label{lem:forget}
Let $D$ be a loopless digraph with adjacency matrix $A$, and let $G$ be the underlying undirected
multigraph obtained by forgetting orientations (that is, each arc of $D$ becomes an undirected edge of $G$).
Then
\[
\mathcal L_{\mathrm{inc}}(D)=L(G)=D_{\mathrm{out}}(D)+D_{\mathrm{in}}(D)-(A+A^{\mathsf T}),\]
where $L(G)$ denotes the Laplacian of the multigraph $G$.
\end{lemma}

\begin{proof}
Let $B=B(D)$ be the oriented incidence matrix of $D$, and let $A_G$ and $D_G$ denote the adjacency matrix and diagonal degree matrix of $G$, respectively. Then $A_G=A+A^{\mathsf T}$ and $(D_G)_{ii}=d_i^+(D)+d_i^-(D)$, hence, $L(G)=D_{\mathrm{out}}(D)+D_{\mathrm{in}}(D)-(A+A^{\mathsf T})$.

We now compute the entries of $BB^{\mathsf T}$.
For $i=j$, each arc incident to $v_i$ contributes $(\pm1)^2=1$ to $(BB^{\mathsf T})_{ii}$, hence
\[
(BB^{\mathsf T})_{ii}=d_i^+(D)+d_i^-(D)=(D_G)_{ii}.
\]
For $i\neq j$, a column (arc) contributes to $(BB^{\mathsf T})_{ij}$ only if it is an arc between $v_i$ and $v_j$.
If $e=(v_i\to v_j)$ then $B_{i,e}B_{j,e}=(-1)(+1)=-1$, and if $e=(v_j\to v_i)$ then $B_{i,e}B_{j,e}=(+1)(-1)=-1$.
Therefore 
\[
(BB^{\mathsf T})_{ij}=-(A_{ij}+A_{ji})=-(A_G)_{ij}.
\]
Thus $BB^{\mathsf T}=D_G-A_G=L(G)$, as required.
\end{proof}
 
\begin{corollary}\label{cor:Eulerian}
Let $D$ be a digraph.
Then $\mathcal {L}_{\mathrm{inc}}(D)-2\,\SSS(D)=D_{\mathrm{in}}(D)-D_{\mathrm{out}}(D)$.
Furthermore, if $D$ is Eulerian, then $\mathcal{L}_{\mathrm{inc}}(D)=2\,\SSS(D)$.
\end{corollary}

\section{Reverse of digraphs}\label{sec:reverse}

Reversing all arcs of a digraph is a simple operation, but it has a useful
effect on the Laplacian matrices associated with the digraph.  In particular,
arc reversal interchanges the roles of in-degrees and out-degrees, and hence
relates the out-degree Laplacian of a digraph to the in-degree Laplacian of its
reverse.  This allows results about one of these Laplacians to be translated
into corresponding results about the other.  We record these relationships for
the Laplacian variants considered in this paper.

\begin{definition}
Let $D$ be a digraph. The \defin{reverse} $\overleftarrow{D}$ of $D$ is the digraph obtained from $D$ by reversing the direction of every arc.
\end{definition}

Since the adjacency matrix satisfies $A(\overleftarrow{D})=A(D)^{\mathsf T}$, the reverse digraph is also referred to as the \emph{transpose} digraph or \emph{dual} digraph in the literature, e.g., see \cite{cavers2025spectra,ostrander2015laplacian}.
The first identities in the following lemma also appear in \cite{ostrander2015laplacian}, 
and the remainder identities follow directly from \cref{lem:prim_lap} and \cref{lem:forget}.
 
\begin{lemma}\label{lem:reverse-all}
Let $D$ be a digraph.
Then the following hold:
\begin{enumerate}
\item 
$\LLL(\overleftarrow D)=\mathcal{L}_{\mathrm{in}}(D)^{\mathsf T}
\quad\text{and}\quad
\mathcal{L}_{\mathrm{in}}(\overleftarrow D)=\LLL(D)^{\mathsf T}$.

\item $\SSS(\overleftarrow D)=\mathcal{S}_{\mathrm{in}}(D)
\quad\text{and}\quad
\mathcal{S}_{\mathrm{in}}(\overleftarrow D)=\SSS(D)$.

\item $\KKK(\overleftarrow D)=-\mathcal{K}_{\mathrm{out}}(D)$.

\item $\mathcal L_{\mathrm{inc}}(\overleftarrow D)=\mathcal L_{\mathrm{inc}}(D)$.
\end{enumerate}
\end{lemma}

With \cref{lem:reverse-all} in mind, we henceforth focus on the out-degree Laplacian variants for the remainder of this paper, that is, $\LLL$, $\SSS$, and $\KKK$, along with the incidence Laplacian $\mathcal L_{\mathrm{inc}}$.

In general, for the reverse digraph, the spectrum of $\LLL(\overleftarrow D)$ and $\LLL(D)$ differ as does the spectrum of $\SSS(\overleftarrow D)$ and $\SSS(D)$. 
However, since $\KKK(D)$ is real skew-symmetric, its spectrum is purely imaginary and symmetric with respect to $0$ (hence invariant under multiplication by $-1$) giving the following corollary.

\begin{corollary}
Let $D$ be a digraph. Then 
\[
\spec(\KKK(\overleftarrow D))=\spec(\KKK(D))
\quad\text{and}\quad
\spec(\mathcal L_{\mathrm{inc}}(\overleftarrow D))=\spec(\mathcal L_{\mathrm{inc}}(D)).
\]
\end{corollary}


\begin{corollary}
Let $D$ be an Eulerian digraph. Then $\SSS(\overleftarrow D)=\SSS(D)$, and hence,
$\spec(\SSS(\overleftarrow D))=\spec(\SSS(D))$.
\end{corollary}

\begin{proof}
By \Cref{cor:Eulerian}, we have $\SSS(D)=\frac{1}{2}\mathcal{L}_{\mathrm{inc}}(D)=\frac{1}{2}\mathcal{L}_{\mathrm{inc}}(\overleftarrow D )=\SSS(\overleftarrow D)$, from which the result follows.
\end{proof}

\subsection{Multiplicity of the zero eigenvalue of $\LLL$}
Note that $0$ is always an eigenvalue of the out-degree and in-degree Laplacians $\LLL$ and $\mathcal{L}_{\mathrm{in}}$.
Its multiplicity depends on the concept of a reach described in \cite{MR2223513}.
Recall that a digraph $D$ is said to be \defin{strongly connected} if, 
for every pair of vertices $u,v \in V(D)$, there exists a directed path from $u$ to $v$.
A \defin{strongly connected component} (or \defin{SCC}) of a directed graph $D$ 
is a maximal strongly connected subdigraph of $D$.
Following \cite{MR2223513,veerman2020primer}, we use the following reach terminology.

\begin{definition}\label{def:reach}
Let $D$ be a digraph with vertex set $V$, and let $\mathcal C(D)$ denote its condensation digraph. A strongly connected component $C$ of $D$ is called a \defin{source SCC} if it has in-degree $0$ in $\mathcal C(D)$, and is called a \defin{sink SCC} if it has out-degree $0$ in $\mathcal C(D)$.
For a source SCC $C$, its corresponding \defin{reach} is the vertex set
\[
R(C)\coloneqq \{\,v\in V:\text{ there exists a directed path from some vertex of }C\text{ to }v\,\}.
\]
Equivalently, the reaches of $D$ are the maximal sets of the form
\[
R(i)\coloneqq \{\,j\in V:\text{ there exists a directed path }i\to j\,\}
\]
for some vertex $i$.
For a sink SCC $C$, its corresponding \defin{co-reach} is the vertex set
\[
R^{-}(C)\coloneqq \{\,v\in V:\text{ there exists a directed path from }v\text{ to some vertex of }C\,\}.
\]
Equivalently, the co-reaches of $D$ are the reaches of the reverse digraph $\overleftarrow D$.
\end{definition}

In \cite{MR2223513}, the multiplicity of $0$ is determined for $\mathcal L_{\mathrm{in}}(D)$, which in turn yields a corresponding result for the multiplicity of $0$ for $\mathcal L_{\mathrm{out}}(D)$.
  
\begin{theorem}{\rm \cite[Thm.~3.2 \& Cor.~4.2]{MR2223513}}\label{thm:zero-mult-reaches}
Let $D$ be a digraph. Then the algebraic and geometric multiplicity of the
eigenvalue $0$ of $\mathcal L_{\mathrm{in}}(D)$ equals the number of reaches of
$D$.
\end{theorem}

\begin{corollary}\label{cor:zero-mult-coreaches}
Let $D$ be a digraph. Then the algebraic and geometric multiplicity of the
eigenvalue $0$ of $\LLL(D)$ equals the number of co-reaches of $D$, equivalently
the number of sink SCCs of $D$.
\end{corollary}

\begin{proof}
By \Cref{lem:reverse-all}, $\LLL(D)^{\mathsf T}=\mathcal L_{\mathrm{in}}(\overleftarrow D)$.
Thus $\LLL(D)$ and $\mathcal L_{\mathrm{in}}(\overleftarrow D)$ have the same
algebraic and geometric multiplicities at $0$. Applying
\Cref{thm:zero-mult-reaches} to $\overleftarrow D$, the multiplicity equals the
number of reaches of $\overleftarrow D$, which is the number of co-reaches of
$D$, equivalently the number of sink SCCs of $D$.
\end{proof}

\Cref{cor:zero-mult-coreaches} implies that if $D$ is a tournament, then $0$ is an eigenvalue of $\mathcal L_{\mathrm{out}}(D)$ with multiplicity one since every tournament has at most one sink vertex and thus exactly one sink SCC.

\section{Complement of digraphs}\label{sec:complement}
For a graph $G$ of order $n$, let $0=\mu_1(G)\leq\mu_2(G)\leq\cdots\leq\mu_n(G)$
denote the eigenvalues of the Laplacian $L(G)$ of $G$.
Since $L(G)+L(\overline{G})=nI-J$, where $\overline{G}$ is the complement of $G$ and $J$ represents the $n\times n$ all ones matrix, it follows that except for the trivial eigenvalue $0$, the Laplacian eigenvalues satisfy
\begin{equation}\label{eq:comp}
\mu_i(G)+\mu_{n+2-i}(\overline{G})=n,    
\end{equation}
for $i=2,\dotsc,n$. 
This classical relation for the complement also holds for $\LLL$ under a suitable ordering of the spectrum.
Recall that a simple digraph has no loops or multiple arcs, but $2$-cycles are permitted.

\begin{definition}\label{def:comp}
Let $D$ be a simple digraph.
The \defin{complement} $\overline{D}$ of $D$ is the digraph with the same vertex set, and with arc $(x,y)\in E(\overline{D})$ if and only if $(x,y)\not\in E(D)$, for all distinct $x,y\in V(D)$.
\end{definition}

Since the adjacency matrix for a complement satisfies
$A(\overline{D})=J-I-A(D)$, the following lemma is immediate.

\begin{lemma}\label{lem:comp-all}
Let $D$ be a digraph on $n$ vertices.
Then the following hold:
\begin{enumerate}
\item $\mathcal{L}_{\mathrm{inc}}(\overline{D})+\mathcal{L}_{\mathrm{inc}}(D)= 2(nI - J)$.
\item $\mathcal{L}_{\mathrm{out}}(\overline{D})+\mathcal{L}_{\mathrm{out}}(D)=nI-J$.
\item $\mathcal{S}_{\mathrm{out}}(\overline{D})+\mathcal{S}_{\mathrm{out}}(D)=nI-J$.
\item $\mathcal{K}_{\mathrm{out}}(\overline{D})=-\mathcal{K}_{\mathrm{out}}(D)$.
\end{enumerate}
\end{lemma}

\begin{proof}
Since $A(\overline D)=J-I-A(D)$ and $D_{\mathrm{out}}(\overline D)=(n-1)I-D_{\mathrm{out}}(D)$, all identities follow by substituting into the formulas in \cref{lem:prim_lap} and \cref{lem:forget} and simplifying.
\end{proof}

The coronal of an $n\times n$ matrix $M$, defined by $\chi_M(x) = \textbf{1}_n^{\mathsf T}\big(x I_n-M\big)^{-1}\textbf{1}_n$ where $\textbf{1}_n$ denotes the all ones vector, often appears in formulas for the characteristic polynomial of complements of graphs and digraphs.
For example, \cite[Theorem~12]{Mcleman} shows
\[
\phi_{A(\overline{G})}(x)=(-1)^n\,\phi_{A(G)}(-x-1)\,\left(1+\chi_{A(G)}(-x-1)\right)
\]
for a graph $G$, and \cite[Corollary~2.4]{cavers2025spectra} shows 
\[
\phi_{A(\overline{D})}(x) = (-1)^n\, \phi_{A(D)}(-x - 1)\,\big(1 + \chi_{A(D)}(-x - 1)\big)
\]
for a digraph $D$ and this also holds if $D$ has multiple arcs. 
More generally, \cite[Theorem~2.2(ii)]{cavers2025spectra} shows that if $M$ is an $n\times n$ matrix and $M'=aM+bJ+cI$, where $a,b,c\in\mathbb{R}$ and $a\neq 0$, then 
\[
\phi_{M'}(x)= a^{n-1}\,\phi_M\left(\tfrac{x-c}{a}\right)\left(a-b\,\chi_M\left(\tfrac{x-c}{a}\right)\right).
\]
We also note that if $M$ is an $n\times n$ matrix with constant row sum $t$, then $\chi_M (x) = n/(x-t)$ (see \cite[Proposition 2]{cui2012spectrum}).
These give a relationship between the characteristic polynomial of a digraph $D$ and of its complement for the Laplacian matrices studied in this paper.

\begin{theorem}\label{lem:comp-poly-variants}
Let $D$ be a simple digraph on $n$ vertices. Then
\begin{enumerate}
\item {\rm \cite{LiSu}}\label{lem:comp}
$\phi_{\LLL(\overline D)}(x)=(-1)^n\,\frac{x}{x-n}\,\phi_{\LLL(D)}(n-x)$.
\item $\phi_{\mathcal L_{\mathrm{inc}}(\overline D)}(x)=(-1)^n\,\frac{x}{x-2n}\,\phi_{\mathcal L_{\mathrm{inc}}(D)}(2n-x)$. 
\item $\phi_{\SSS(\overline D)}(x)=(-1)^n\,\phi_{\SSS(D)}(n-x)\,\big(1-\chi_{\SSS(D)}(n-x)\big)$.
\item $\phi_{\KKK(\overline D)}(x)=\phi_{\KKK(D)}(x)$.
\end{enumerate}
\end{theorem}

\begin{proof}
We have
$\mathcal{S}_{\mathrm{out}}(\overline D)=-\mathcal{S}_{\mathrm{out}}(D)+nI-J$ and $\mathcal{L}_{\mathrm{inc}}(\overline D)=-\mathcal{L}_{\mathrm{inc}}(D)+2nI-2J$ by Lemma~\ref{lem:comp-all}.
Hence the formulas for $\phi_{\SSS(\overline D)}(x)$ and
$\phi_{\mathcal L_{\mathrm{inc}}(\overline D)}(x)$ follow directly from
\cite[Theorem~2.2]{cavers2025spectra} together with $\chi_{\mathcal{L}_{\mathrm{inc}}(D)}(x)=n/x$ since $\mathcal{L}_{\mathrm{inc}}(D)$ has constant row sum zero.
Finally, since $\KKK(\overline D)=-\KKK(D)$ by \Cref{lem:prim_lap}, we have
\[
\phi_{\KKK(\overline D)}(x)=\det(xI+\KKK(D))
=(-1)^n\phi_{\KKK(D)}(-x)
=\phi_{\KKK(D)}(x).
\]
The final equality follows since $\KKK(D)$ is real skew-symmetric and its spectrum is symmetric about $0$.
\end{proof}
 
If $D$ is Eulerian, then $\mathcal{S}_{\mathrm{out}}(D)$ has constant row sum and thus $\chi_{\mathcal{S}_{\mathrm{out}}(D)}(x)=n/x$ giving the following simplification for $\phi_{\SSS(\overline D)}(x)$ for Eulerian digraphs.

\begin{corollary}\label{cor:eulSSS}
If $D$ is a simple Eulerian digraph, then
\[
\phi_{\SSS(\overline D)}(x)=(-1)^n\,\frac{x}{x-n}\,\phi_{\SSS(D)}(n-x).
\]
\end{corollary}

\Cref{lem:comp-poly-variants} and \Cref{cor:eulSSS} show that an analogue to the classical relation \cref{eq:comp} for graph Laplacians also holds for digraph Laplacians.
For multisets $A$ and $B$, $A\sqcup B$ denotes the disjoint union of $A$ and $B$.

\begin{corollary}\label{cor:compeig}
Let $D$ be a simple digraph on $n$ vertices and $\sm_1\{0\}$ denote deleting exactly one copy of $0$ from the multiset.
Then the following hold:
\begin{enumerate}
\item 
$\spec\!\big({\mathcal L}_{\mathrm{out}}(\overline D)\big)
=\{0\}
\;\sqcup\;
\left(n-\big(\spec({\mathcal L}_{\mathrm{out}}(D))\sm_1\{0\}\big)\right)$.
\item 
$\spec\!\big({\mathcal L}_{\mathrm{inc}}(\overline D)\big)
=\{0\}
\;\sqcup\;
\left(2n-\big(\spec({\mathcal L}_{\mathrm{inc}}(D))\sm_1\{0\}\big)\right)$.
\item If $D$ is Eulerian, then\\
$\spec\!\big({\mathcal S}_{\mathrm{out}}(\overline D)\big)
=\{0\}
\;\sqcup\;
\left(n-\big(\spec({\mathcal S}_{\mathrm{out}}(D))\sm_1\{0\}\big)\right)$.
    \item $\spec(\KKK(\overline D))=\spec(\KKK(D))$. 
\end{enumerate}
\end{corollary}

In the case that $D$ is not Eulerian, a relation similar to \cref{eq:comp} holds for some of the eigenvalues of $\SSS$.

\begin{corollary}
Let $D$ be a simple digraph on $n$ vertices
and suppose $v$ is an eigenvector for $\SSS(D)$ with eigenvalue $\mu$ such that $v\perp \mathbf 1$. Then $v$ is an eigenvector for $\SSS(\overleftarrow D)$ with eigenvalue $n-\mu$.
\end{corollary}

\section{Line Digraphs}\label{sec:line_digraphs}

Line digraphs exhibit several intriguing properties, particularly regarding their spectra.
 
\begin{definition}
Let $D=(V,E)$ be a digraph. 
The \defin{line digraph} $L_D$ of $D$ has vertex set $V(L_D)\coloneqq E$, and there is an arc from $(i,j)$ to $(j,k)$ in $L_D$ whenever $(i,j),(j,k)\in E$.
\end{definition}

One notable characteristic is the frequent appearance of the eigenvalue 0, often with high multiplicity. 

\begin{theorem}{\rm\cite[Theorem 2.3]{MR2312328}}
Let $D$ be a digraph of order $n$ and size $m$, and $L_D$ its line digraph.  
Then
\[
\phi_{A(L_D)}(x)=x^{m-n}\phi_{A(D)}(x).
\]
\end{theorem}

An analogous result does not apply to the out-degree Laplacian matrix as demonstrated in the following example.
First we introduce abbreviated notation for the characteristic polynomial of the out-degree Laplacian of a digraph $D$.

\begin{notation}\rm
For a digraph $D$, let $\phi_D(x)\coloneqq\phi_{\mathcal L_{\mathrm{out}}(D)}(x)$.
\end{notation}

\begin{example}\label{ex:line}\rm
Consider the digraph in \cref{fig:enter-label}.
\begin{figure}[!ht]
    \centering
    \subfloat[\centering ]{\includegraphics[scale=0.75]{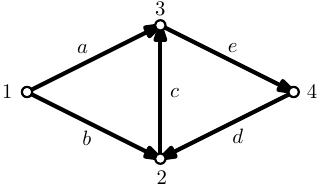} }%
    \qquad
    \subfloat[\centering ]{\includegraphics[scale=0.75]{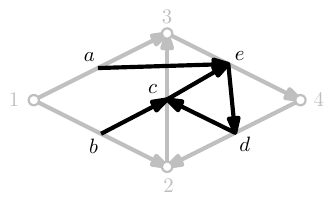} }
    \caption{A digraph of order 4 with its corresponding line digraph}
    \label{fig:enter-label}
\end{figure}
Order the vertices of $D$ as $1,2,3,4$ and the vertices of $L_D$ (that is, the arcs of $D$) as
\[
(1,3),\ (1,2),\ (2,3),\ (4,2),\ (3,4),
\]
we have
\[
A(D) = \begin{bmatrix}
 0 & 1 & 1 & 0 \\
0 & 0 & 1 & 0 \\
0 & 0 & 0 & 1 \\
0 & 1 & 0 & 0
\end{bmatrix},\qquad
A(L_D) = \begin{bmatrix}
 0 & 0 & 0 & 0 & 1 \\
0 & 0 & 1 & 0 & 0 \\
0 & 0 & 0 & 0 & 1 \\
0 & 0 & 1 & 0 & 0 \\
0 & 0 & 0 & 1 & 0
\end{bmatrix}.
\]
Moreover,
\[
{\mathcal L}_{\mathrm{out}}(D)=D_{\mathrm{out}}(D)-A(D) = \begin{bmatrix}
 2 & -1 & -1 & 0 \\
0 & 1 & -1 & 0 \\
0 & 0 & 1 & -1 \\
0 & -1 & 0 & 1
\end{bmatrix},
\]
and
\[
{\mathcal L}_{\mathrm{out}}(L_D)=D_{\mathrm{out}}(L_D)-A(L_D) = \begin{bmatrix}
 1 & 0 & 0 & 0 & -1 \\
0 & 1 & -1 & 0 & 0 \\
0 & 0 & 1 & 0 & -1 \\
0 & 0 & -1 & 1 & 0 \\
0 & 0 & 0 & -1 & 1
\end{bmatrix}.
\]
Hence $\phi_D(x) = x^4-5x^3+9x^2-6x$, and $\phi_{L_D}(x) = x^5 - 5x^4 + 10x^3 - 9x^2 + 3x$. Thus, $\phi_{L_D}(x)\neq x^{m-n}\,\phi_D(x)$, so the adjacency relation
does not extend in the same simple zero-padding form to the out-degree Laplacian.
However, we observe that $\phi_{L_D}(x)=\frac{(x-1)^2}{x-2}\phi_D(x)$ and prove in \cref{thm:outlap-line-general} a relationship between $\phi_D$ and $\phi_{L_D}$ in the general case involving the in-degrees and out-degrees of the vertices of $D$. 
\end{example}

Let $D=(V,E)$ be a digraph with $|V|=n$ and $|E|=m$.
Index the columns of the following matrices by arcs $e\in E$ and rows by vertices $v\in V$.

\noindent The \defin{out-incidence matrix} $B^{+}\in\{0,1\}^{n\times m}$ is defined by
\[
(B^{+})_{v,e}=
\begin{cases}
1 & \text{if $e=(v,w)$ for some $w$},\\
0 & \text{otherwise.}
\end{cases}
\]
Similarly, the \defin{in-incidence matrix} $B^{-}\in\{0,1\}^{n\times m}$ is defined by
\[
(B^{-})_{v,e}=
\begin{cases}
1 & \text{if $e=(w,v)$ for some $w$},\\
0 & \text{otherwise.}
\end{cases}
\]

\begin{lemma}{\rm\cite[Lemma 2.2]{MR2312328}}\label{bb}
Let $D$ be a digraph of order $n$, size $m$, and with adjacency matrix $A(D)$, and let $L_D$ be its line digraph.
Let $B^+$ and $B^-$ be the out-incidence and in-incidence matrices of $D$, respectively.
Then
\[
A(D)=B^{+}(B^{-})^{\mathsf T}
\qquad\text{and}\qquad
A(L_D)=(B^{-})^{\mathsf T}B^{+}.
\]
\end{lemma}

We now state the matrix determinant lemma.

\begin{lemma}{\rm\cite{MR2978290}}\label{lem:det-lemma}
Let $X\in\mathbb F^{m\times m}$ be invertible,
$U\in\mathbb F^{m\times r}$, and $V\in\mathbb F^{r\times m}$.
Then
\[
\det(X+UV)=\det(X)\det(I_r+VX^{-1}U).
\]
In particular, for $X=I_m$, $\det(I_m+UV)=\det(I_r+VU)$.
\end{lemma}

\begin{proposition}\label{thm:outlap-line-general}
Let $D=(V,E)$ be a digraph of order $n$ and size $m$, and $L_D$ its line digraph.  
Then
\[
\phi_{L_D}(x)\prod_{v\in V}(x-d_v^+)
=
\phi_D(x)\prod_{v\in V}(x-d_v^+)^{d_v^-}.
\]
\end{proposition}

\begin{proof}
We set $B\coloneqq B^+$, $C\coloneqq (B^-)^{\mathsf T}$.
Then, by \Cref{bb}, $A(D)=BC,$ and $A(L_D)=CB$.
Let $D_1\coloneqq D_{\mathrm{out}}(D)$ and $D_2\coloneqq D_{\mathrm{out}}(L_D)$.
Thus
\[
\LLL(D)=D_1-BC,\qquad \LLL(L_D)=D_2-CB.
\]
For an arc $e=(u,v)$ of $D$, the out-degree of the corresponding vertex $e$ in
$L_D$ is $d_v^+$. Hence, for every arc $e$ and vertex $w$,
\[
(D_2C)_{e,w}
=d_{\operatorname{head}(e)}^+\,C_{e,w}
=C_{e,w}\,d_w^+
=(CD_1)_{e,w}.
\]
Therefore $D_2C=CD_1$.
Now we set $P\coloneqq xI_n-D_1$, and $Q\coloneqq xI_m-D_2$.
The relation $D_2C=CD_1$ gives $QC=CP$.
For all $x$ for which $P$ and $Q$ are invertible, this implies
$Q^{-1}C=CP^{-1}$.
Therefore, we obtain
\begin{align*}
\phi_{L_D}(x)
&=\det(xI_m-\LLL(L_D))\\
&=\det(Q+CB)\\
&=\det(Q)\det(I_m+Q^{-1}CB)\\
&=\det(Q)\det(I_m+CP^{-1}B)\\
&=\det(Q)\det(I_n+BCP^{-1})\\
&=\det(Q)\frac{\det(P+BC)}{\det(P)}\\
&=\frac{\det(xI_m-D_2)}{\det(xI_n-D_1)}\,\phi_D(x).
\end{align*}
Multiplying by $\det(xI_n-D_1)$ gives
$\phi_{L_D}(x)\det(xI_n-D_1)
=
\phi_D(x)\det(xI_m-D_2)$.
Both sides are polynomials in $x$, so the identity holds for all $x$.
Finally,
\[
\det(xI_n-D_1)=\prod_{v\in V}(x-d_v^+),
\]
and
\[
\det(xI_m-D_2)
=
\prod_{e=(u,v)\in E}(x-d_v^+)
=
\prod_{v\in V}(x-d_v^+)^{d_v^-}.
\]
Substituting these two products proves the claim.
\end{proof}


\begin{corollary}\label{thm:outlap-line-regular-phi}
Let $D$ be a digraph on $n$ vertices and $m\geq 1$ arcs, and suppose $D$ is
$r$-out-regular. Then $\phi_{L_D}(x)=(x-r)^{m-n}\,\phi_D(x)$.
\end{corollary}

\begin{corollary}
Let $D$ be an Eulerian digraph on vertices $v_1,\dots,v_n$ with at least one arc, and write
$d_i=d_{v_i}^+=d_{v_i}^-$. Then
\[
\phi_{L_D}(x)=\phi_D(x)\prod_{i=1}^n(x-d_i)^{d_i-1}.
\]
\end{corollary}

\begin{corollary}
Let $D$ be a tournament on vertices $v_1,\dots,v_n$, and write $d_i=d_{v_i}^+$. Then
\[
\phi_{L_D}(x)=\phi_D(x)\prod_{i=1}^n(x-d_i)^{n-d_i-2}.
\]
\end{corollary}




We next demonstrate that the spectrum of the remaining digraph Laplacians $\SSS$, $\KKK$ and $\mathcal L_{\mathrm{inc}}$ of a line digraph splits into an two parts: one consisting of ``degrees'' and one coming from the incidence interaction from the original digraph.

\begin{proposition}\label{prop:sss-line-2n}
Let $D=(V,E)$ be a digraph of order $n$ and size $m$ 
with $B^+$ and $B^-$ its out-incidence and in-incidence matrices, respectively.
Define the diagonal matrix $X_{\SSS}\in\mathbb{R}^{m\times m}$ by
\[
(X_{\SSS})_{e,e}\coloneqq x-d^+_{\operatorname{head}(e)}.
\] 
Then 
\[
\phi_{\SSS(L_D)}(x)
=\det(X_{\SSS})\,
\det\Bigl(I_{2n}+\tfrac12
\begin{bmatrix} B^+ \\[1mm] B^- \end{bmatrix}
X_{\SSS}^{-1}
\begin{bmatrix} (B^-)^{\mathsf T} & (B^+)^{\mathsf T} \end{bmatrix}
\Bigr),
\]
where
\[
\det(X_{\SSS})
=\prod_{v\in V(D)} (x-d_v^+)^{d_v^-}.\]
\end{proposition}

\begin{proof}
By definition, $\SSS(L_D)=D_{\mathrm{out}}(L_D)-\tfrac12\bigl(A(L_D)+A(L_D)^{\mathsf T}\bigr)$.
For an arc $e=(u\to v)$, the out-degree of $e$ in $L_D$ equals $d_v^+$, so
$xI_m-D_{\mathrm{out}}(L_D)=X_{\SSS}$.
By \cref{bb}, $A(L_D)=(B^-)^{\mathsf T}B^+$ and $A(L_D)^{\mathsf T}=(B^+)^{\mathsf T}B^-$, hence $xI_m-\SSS(L_D)
= X_{\SSS}+\tfrac12\Bigl((B^-)^{\mathsf T}B^+ + (B^+)^{\mathsf T}B^-\Bigr)
= X_{\SSS}+UV$, where
\[
U=\begin{bmatrix} (B^-)^{\mathsf T} & (B^+)^{\mathsf T}\end{bmatrix}\in\mathbb{R}^{m\times 2n},
\qquad
V=\tfrac12\begin{bmatrix} B^+ \\[1mm] B^- \end{bmatrix}\in\mathbb{R}^{2n\times m}.
\]
Applying~\Cref{lem:det-lemma} with $X=X_{\SSS}$ gives the stated formula.
The computation of $\det(X_{\SSS})$ is immediate from the definition of $X_{\SSS}$.
\end{proof}


\begin{proposition}\label{prop:kkk-line-2n}
Let $D$ be a digraph of order $n$ and size $m$ with $B^+$ and $B^-$ its out-incidence and in-incidence matrices. Then
\[
\phi_{\KKK(L_D)}(x)=x^{m-2n}\,
\det\Bigl(xI_{2n}+\frac12
\begin{bmatrix} B^+ \\[1mm] -B^- \end{bmatrix}
\begin{bmatrix} (B^-)^{\mathsf T} & (B^+)^{\mathsf T} \end{bmatrix}
\Bigr).
\]
\end{proposition}

\begin{proof}
Since $\KKK(L_D)=\tfrac12\bigl(A(L_D)^{\mathsf T}-A(L_D)\bigr)$ and
$A(L_D)=(B^-)^{\mathsf T}B^+$, we have
\[
xI_m-\KKK(L_D)
= xI_m+\tfrac12\Bigl((B^-)^{\mathsf T}B^+-(B^+)^{\mathsf T}B^-\Bigr)
= xI_m+UV,
\]
where
\[
U=\begin{bmatrix} (B^-)^{\mathsf T} & (B^+)^{\mathsf T}\end{bmatrix},
\qquad
V=\tfrac12\begin{bmatrix} B^+ \\[1mm] -B^- \end{bmatrix}.
\]
Applying~\Cref{lem:det-lemma} with $X=xI_m$ gives the stated formula.
\end{proof}


\begin{proposition}\label{prop:inc-line-2n}
Let $D=(V,E)$ be a digraph 
with $B^+$ and $B^-$ its out-incidence and in-incidence matrices, respectively.
For an arc $e=(u\to v)\in E$ set
$t_e\coloneqq d_v^++d_u^-$.
Define the diagonal matrix $X_{\mathrm{inc}}\in\mathbb{R}^{m\times m}$ by $(X_{\mathrm{inc}})_{e,e}\coloneqq x-t_e$.
Then
\[
\phi_{\mathcal{L}_\mathrm{inc}(L_D)}(x)
=\det(X_{\mathrm{inc}})\,
\det\Bigl(I_{2n}+
\begin{bmatrix} B^+ \\[1mm] B^- \end{bmatrix}
X_{\mathrm{inc}}^{-1}
\begin{bmatrix} (B^-)^{\mathsf T} & (B^+)^{\mathsf T} \end{bmatrix}
\Bigr),
\]
where $\det(X_{\mathrm{inc}})=\prod_{e\in E(D)}(x-t_e)$.
\end{proposition}

\begin{proof}
By definition, we have $\mathcal L_{\mathrm{inc}}(L_D)=D_{\mathrm{total}}(L_D)-\bigl(A(L_D)+A(L_D)^{\mathsf T}\bigr)$, where $D_{\mathrm{total}}=D_{\mathrm{out}}+D_{\mathrm{in}}$.
For an arc $e=(u\to v)\in E$ where $u,v\in V$, the out-degree of $e$ in $L_D$ is $d_v^+$ and the in-degree is $d_u^-$,
so the total degree is $t_e=d_v^++d_u^-$, and therefore $xI_m-D_{\mathrm{tot}}(L_D)=X_{\mathrm{inc}}$.
Hence
\[
xI_m-\mathcal L_{\mathrm{inc}}(L_D)
= X_{\mathrm{inc}}+\Bigl((B^-)^{\mathsf T}B^+ + (B^+)^{\mathsf T}B^-\Bigr)
= X_{\mathrm{inc}}+UV,
\]
with
\[
U=\begin{bmatrix} (B^-)^{\mathsf T} & (B^+)^{\mathsf T}\end{bmatrix},
\qquad
V=\begin{bmatrix} B^+ \\[1mm] B^- \end{bmatrix}.
\]
Applying \Cref{lem:det-lemma} with $X=X_{\mathrm{inc}}$ gives the stated formula; $\det(X_{\mathrm{inc}})=\prod_{e}(x-t_e)$
follows from the definition of $X_{\mathrm{inc}}$.
\end{proof}

\section{Joins}\label{sec:join}

Here we analyze the out-degree Laplacian variants for the join of two digraphs.

\begin{definition}
Let $D_1=(V_1,E_1)$ and $D_2=(V_2,E_2)$ be two digraphs.
The \defin{direct sum} of $D_1$ and $D_2$, denoted $D=D_1\dotplus D_2$, is the digraph with $V(D)=V_1\sqcup V_2$ and $E(D)=E_1\cup E_2$ (where $V_1\cap V_2=\emptyset$).
The \defin{join} (or \defin{complete product}) of $D_1$ and $D_2$, denoted $D=D_1\triangledown D_2$, is obtained from $D_1\dotplus D_2$ by adding, for each $u \in V_1$ and $v \in V_2$, the two arcs $u \to v$ and $v \to u$ (that is, a $2$-cycle between every vertex of $D_1$ and every vertex of $D_2$).
\end{definition}

\begin{figure}[!ht]
\centering

\begin{subfigure}[t]{0.48\textwidth}
\centering
\begin{tikzpicture}[
  >=Stealth,
  vtx/.style={circle,draw,fill=white,inner sep=0pt,minimum size=5pt},
  arc/.style={->,line width=1.0pt,shorten >=2pt}
]
\coordinate (a1) at (0,1.2);
\coordinate (a2) at (0,0);
\coordinate (a3) at (0,-1.2);
\draw[arc] (a1)--(a2);
\draw[arc] (a2)--(a3);
\draw[arc] (a1)--(a3);
\node[vtx] (a1) at (0,1.2) {};
\node[vtx] (a2) at (0,0)   {};
\node[vtx] (a3) at (0,-1.2){};

\coordinate (b1) at (3.2,0.8);
\coordinate (b2) at (3.2,-0.8);
\draw[arc] (b1)--(b2);
\draw[arc] (b2)--(b1);
\node[vtx] (b1) at (3.2,0.8)  {};
\node[vtx] (b2) at (3.2,-0.8) {};

\end{tikzpicture}
\caption{The direct sum $D_1\dotplus D_2$.}
\end{subfigure}
\hfill
\begin{subfigure}[t]{0.48\textwidth}
\centering
\begin{tikzpicture}[
  >=Stealth,
  vtx/.style={circle,draw,fill=white,inner sep=0pt,minimum size=5pt},
  arc/.style={->,line width=1.0pt,shorten >=1pt}
]
\coordinate (a1) at (0,1.2);
\coordinate (a2) at (0,0);
\coordinate (a3) at (0,-1.2);
\coordinate (b1) at (3.2,0.8);
\coordinate (b2) at (3.2,-0.8);

\foreach \u in {a1,a2,a3}{%
  \foreach \v in {b1,b2}{%
    \draw[->,line width=1.0pt,shorten >=6pt] (\u)--(\v);%
    \draw[->,line width=1.0pt,shorten >=3pt] (\v)--(\u);%
  }%
}%

\draw[arc] (a1)--(a2);
\draw[arc] (a2)--(a3);
\draw[arc] (a1)--(a3);

\draw[arc] (b1)--(b2);
\draw[arc] (b2)--(b1);

\node[vtx] (a1) at (0,1.2) {};
\node[vtx] (a2) at (0,0)   {};
\node[vtx] (a3) at (0,-1.2){};

\node[vtx] (b1) at (3.2,0.8)  {};
\node[vtx] (b2) at (3.2,-0.8) {};

\end{tikzpicture}
\caption{The join $D_1\triangledown D_2$.}
\end{subfigure}

\caption{The direct sum and join of two digraphs $D_1$ and $D_2$, where $D_1$ is a directed path on three vertices and $D_2$ is a directed cycle on two vertices.}
\label{fig:sum-join}
\end{figure}
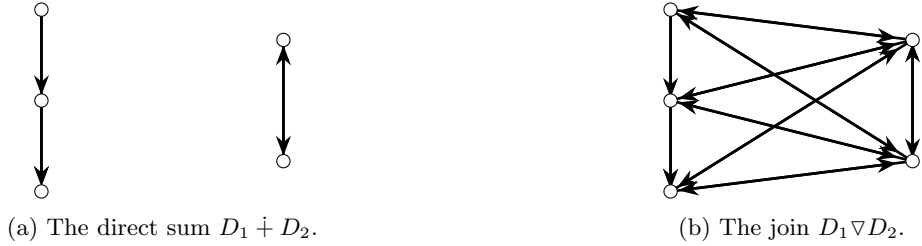

\cref{fig:sum-join} shows the direct sum and join of a directed path on three vertices and a directed cycle on two vertices.
To derive explicit expressions for the digraph Laplacian characteristic polynomials of a join of digraphs, we
need the following lemmas.
We use $J_{n_1,n_2}$ to denote the $n_1\times n_2$ all-ones matrix and $J_{n}\coloneqq J_{n,n}$.

\begin{lemma}{\rm\cite{meyer2023matrix}}
\label[lemma]{schur}
Let $ M_1, M_2, M_3 $, and $ M_4 $ be respectively $ p \times p $, $ p \times q $, $ q \times p $, and $ q \times q $ matrices.
If $M_4$ is invertible, then
\[\det \begin{bmatrix} 
    M_1 & M_2 \\
    M_3 & M_4 
\end{bmatrix} = \det(M_4) \cdot \det \left( M_1 - M_2 M_4^{-1} M_3 \right).\]
\end{lemma}
 
\begin{lemma}{\rm\cite{meyer2023matrix}}\label[lemma]{sylv}
Let $C$ be an $n\times n$ invertible matrix and $\alpha\in\mathbb{R}$. Then 
$\displaystyle\det(C+\alpha J_n)=\det(C)\,(1+\alpha\,\textnormal{\textbf{1}}_n^TC^{-1}\textnormal{\textbf{1}}_n)$.
\end{lemma}

The following lemma is immediate.

\begin{lemma}\label{lem:block}
Let $a,b,c\in\mathbb{R}$, $A$ be an $n_1\times n_1$ matrix and $B$ an $n_2\times n_2$ matrix.
Let
\[
M=
\begin{bmatrix}
A+aI_{n_1} & cJ_{n_1,n_2}\\
cJ_{n_2,n_1} & B+bI_{n_2}
\end{bmatrix}.
\]
Then
$\phi_M(x)= \phi_A(x-a)\,\phi_B(x-b)\, \Big(1-c^2\,\chi_A(x-a)\,\chi_B(x-b)\Big)$.
\end{lemma}
 
\begin{proof}
By \cref{schur},
\begin{align*}
\phi_M(x)
&=
\det\begin{bmatrix}
(x-a)I_{n_1}-A & -cJ_{n_1,n_2}\\
-cJ_{n_2,n_1} & (x-b)I_{n_2}-B
\end{bmatrix} \\
&=
\phi_B(x-b)\,
\det\!\Big((x-a)I_{n_1}-A - c^2J_{n_1,n_2}((x-b)I_{n_2}-B)^{-1}J_{n_2,n_1}\Big) \\
&=
\phi_B(x-b)\,
\det\!\Big((x-a)I_{n_1}-A - c^2\chi_B(x-b)\,J_{n_1}\Big) \\
&=
\phi_B(x-b)\,\phi_A(x-a)\,
\Big(1-c^2\chi_A(x-a)\chi_B(x-b)\Big),
\end{align*}
where we used
\[
J_{n_1,n_2}((x-b)I_{n_2}-B)^{-1}J_{n_2,n_1}
=
\chi_B(x-b)\,J_{n_1}
\]
and \cref{sylv}.
\end{proof}


\begin{theorem}\label{lem:join-poly-variants}
Let $D=D_1\triangledown D_2$, where $D_1$ and $D_2$ are simple digraphs on $n_1\geq 1$ and $n_2\geq 1$ vertices, respectively, and let $n=n_1+n_2$. Then  
\begin{enumerate}
\item {\rm \cite{LiSu}} 
$\phi_{\LLL(D)}(x)=\frac{x(x-n)}{(x-n_1)(x-n_2)}\,\phi_{\LLL(D_1)}(x-n_2)\,\phi_{\LLL(D_2)}(x-n_1)$.
\item $\phi_{\mathcal L_{\mathrm{inc}}(D)}(x)=
\frac{x(x-2n)}{(x-2n_1)(x-2n_2)}\,
\phi_{\mathcal L_{\mathrm{inc}}(D_1)}(x-2n_2)\,
\phi_{\mathcal L_{\mathrm{inc}}(D_2)}(x-2n_1).$
\item $\phi_{\SSS(D)}(x)=\phi_{\SSS(D_1)}(x-n_2)\,\phi_{\SSS(D_2)}(x-n_1)\left(
1-\chi_{\SSS(D_1)}(x-n_2)\,\chi_{\SSS(D_2)}(x-n_1)
\right)$.
\item $\phi_{\KKK(D)}(x)=\phi_{\KKK(D_1)}(x)\,\phi_{\KKK(D_2)}(x)$.
\end{enumerate}
\end{theorem}

\begin{proof}
Since the join adds all arcs in both directions between $V(D_1)$ and $V(D_2)$, with respect to the decomposition $V(D)=V(D_1)\sqcup V(D_2)$, we have
$D_{\mathrm{out}}(D)=(D_{\mathrm{out}}(D_1)+n_2I_{n_1})\oplus(D_{\mathrm{out}}(D_2)+n_1I_{n_2})$, and
$D_{\mathrm{in}}(D)=(D_{\mathrm{in}}(D_1)+n_2I_{n_1})\oplus(D_{\mathrm{in}}(D_2)+n_1I_{n_2})$, and
\[
A(D)=
\begin{bmatrix}
A(D_1) & J_{n_1,n_2}\\
J_{n_2,n_1} & A(D_2)
\end{bmatrix},
\]
where $M_1\oplus M_2$ denotes the direct sum of matrices $M_1$ and $M_2$.
By \cref{lem:forget}, 
\[
\mathcal L_{\mathrm{inc}}(D)
=\begin{bmatrix}
\mathcal L_{\mathrm{inc}}(D_1)+2n_2I_{n_1} & -2J_{n_1,n_2}\\
-2J_{n_2,n_1} & \mathcal L_{\mathrm{inc}}(D_2)+2n_1I_{n_2}
\end{bmatrix}.
\]
Since $\mathcal{L}_{\mathrm{inc}}$ has constant row sum zero,
the second equation follows from \cref{lem:block} using
$\chi_{\mathcal{L}_{\mathrm{inc}}(D_1)}(x)=n_1/x$ 
and 
$\chi_{\mathcal{L}_{\mathrm{inc}}(D_2)}(x)=n_2/x$.

Finally, \cref{lem:prim_lap} gives
\[
\SSS(D)
=
\begin{bmatrix}
\SSS(D_1)+n_2I_{n_1} & -J_{n_1,n_2}\\
-J_{n_2,n_1} & \SSS(D_2)+n_1I_{n_2}
\end{bmatrix},
\quad
\KKK(D)
=\begin{bmatrix}
\KKK(D_1) & O\\
O & \KKK(D_2)
\end{bmatrix},
\] 
and the third and fourth equations follow by \cref{lem:block}.
\end{proof}

Recall from earlier that if $D$ is Eulerian, then $\chi_{\mathcal{S}_{\mathrm{out}}(D)}(x)=n/x$ giving the following simplification for $\phi_{\SSS(D_1\triangledown D_2)}(x)$ for Eulerian digraphs.

\begin{corollary}
Let $D=D_1\triangledown D_2$, where $D_1$ and $D_2$ are simple digraphs on $n_1\geq 1$ and $n_2\geq 1$ vertices, respectively, and let $n=n_1+n_2$. 
If $D_1$ and $D_2$ are Eulerian,
then $D$ is Eulerian and
\[
\phi_{\SSS(D)}(x)=\frac{x(x-n)}{(x-n_1)(x-n_2)}\,\phi_{\SSS(D_1)}(x-n_2)\,\phi_{\SSS(D_2)}(x-n_1).
\]
\end{corollary}

\begin{corollary}\label{cor:joinspec}
Let $D=D_1\triangledown D_2$, where $D_1$ and $D_2$ are simple digraphs on $n_1\geq 1$ and $n_2\geq 1$ vertices, respectively, and let $n\coloneqq n_1+n_2$
and $\sm_1\{0\}$ denote deleting exactly one copy of $0$ from the multiset.
Then the following hold:
\begin{enumerate}


\item 
$\spec\!\big({\mathcal L}_{\mathrm{out}}(D)\big)
=\{0,n\}
\;\sqcup\;
\left(\big(\spec({\mathcal L}_{\mathrm{out}}(D_1))\sm_1\{0\}\big)+n_2\right)
\;\sqcup\;
\left(\big(\spec({\mathcal L}_{\mathrm{out}}(D_2))\sm_1\{0\}\big)+n_1\right).
$

\item 
$\spec\!\big(\mathcal L_{\mathrm{inc}}(D)\big)
=
\{0,2n\}
\;\sqcup\;
\left(\Big(\spec\!\big(\mathcal L_{\mathrm{inc}}(D_1)\big)\sm_1\{0\}\Big)+2n_2\right)
\;\sqcup\;
\left(\Big(\spec\!\big(\mathcal L_{\mathrm{inc}}(D_2)\big)\sm_1\{0\}\Big)+2n_1\right)$.

\item If $D_1$ and $D_2$ are Eulerian digraphs, then\\
$
\spec\big(\SSS(D)\big)
=
\{0,n\}
\;\sqcup\;
\left(\big(\spec(\SSS(D_1))\sm_1\{0\}\big)+n_2\right)
\;\sqcup\;
\left(\big(\spec(\SSS(D_2))\sm_1\{0\}\big)+n_1\right).
$
\item $\spec\big(\KKK(D)\big)
=\spec\big(\KKK(D_1)\big)\;\sqcup\;\spec\big(\KKK(D_2)\big)$.
\end{enumerate}
\end{corollary}

\subsection{Spectral radius of $\LLL$}\label{sec:specradius}
For undirected graphs, the
Laplacian spectral radius is bounded above by number of vertices, and this suggests asking whether an analogous bound holds for the possibly complex eigenvalues of $\LLL(D)$.  
Although $\LLL(D)$ need not be symmetric and its
eigenvalues need not be real, 
the same bound on its spectral radius does indeed hold.




\begin{theorem}\label{thm:Lout-bound}
Let $D$ be a simple digraph of order $n\geq 2$. 
If $\mu\in\spec(\LLL(D))$ then $|\mu|\leq n$, and equality occurs only when $\mu=n$.
\end{theorem}

\begin{proof}
Let $\mu\in\spec(\mathcal L_{\mathrm{out}}(D))$. 
We claim that $\mu\in\{z\in\mathbb C:|z-(n-1)|\leq n-1\}$.
If $\mu=0$ the claim is immediate, so assume $\mu\neq 0$.
By Gershgorin's circle theorem, $\mu$ lies in at least one disc
\[
\{z\in\mathbb C:|z-d_i^+|\leq d_i^+\},
\]
for some $i=1,2,\ldots, n$.
Since $0\leq d_i^+\leq n-1$ for simple digraphs, we have
\[
|\mu-(n-1)|
\leq |\mu-d_i^+|+|d_i^+-(n-1)|
\leq d_i^+ + (n-1-d_i^+)
= n-1,
\]
and hence, $\mu\in\{z\in\mathbb C:|z-(n-1)|\leq n-1\}$ as required.

By \cref{cor:compeig}, $n-\mu$ is an eigenvalue of $\overline D$.
Since $\overline D$ is also a simple digraph, it follows from the claim that 
$n-\mu\in\{z\in\mathbb C:|z-(n-1)|\leq n-1\}$.
Hence $|(n-\mu)-(n-1)|\leq n-1$ giving $|\mu-1|\leq n-1$.
By the triangle-inequality, we have
\[
|\mu|\leq|\mu-1|+1\leq n,
\] 
as required.
Recall that for $z_1,z_2\neq 0$, equality occurs in the triangle inequality $|z_1+z_2|\leq|z_1|+|z_2|$ if and only if $z_1=tz_2$ for some $t\geq 0$.
Thus, in the case that $|\mu|=n$ we must have $\mu\in\mathbb{R}$, hence, $\mu=n$.
\end{proof}

\begin{remark}\rm
We note that if $D$ is an oriented graph, then the stronger bound 
$|\mu|\le n-1$ follows immediately from Gershgorin's theorem applied to
$\mathcal L_{\mathrm{out}}(D)^{\mathsf T}$. Indeed, for some vertex $v_i$,
\[
    |\mu-d_i^+|\le d_i^-,
\]
and therefore
\[
    |\mu|\le d_i^+ + d_i^- \le n-1.
\]
This stronger bound is false for general simple digraphs, since $2$-cycles may occur. In fact, equality $|\mu|=n$ is possible in that setting as demonstrated below.
\end{remark}

In the undirected setting where $G$ is a graph of order $n$, one has $n \in \spec(L(G))$ if and only if $G$ is the join of two graphs. This follows from the fact that the Laplacian spectral radius satisfies $\rho(L(G)) \le n$, with equality if and only if $\overline{G}$ is disconnected. Moreover, $\overline{G}$ is disconnected precisely when $G$ can be expressed as the join of two graphs.

In the digraph setting, if $D$ is a digraph on $n$ vertices and is the join of two digraphs, then \cref{cor:joinspec} shows that $n$ is an eigenvalue of $\LLL(D)$.
However, the converse of this statement is false.
For example, the digraph with the following out-degree Laplacian is not a join, but has $n=4$ as an eigenvalue:
\[
L=\begin{bmatrix}
 1 & -1 & 0 & 0\\ -1 & 3 & -1 & -1\\ 0 & 0 & 1 & -1\\ -1 & -1 & -1 & 3 \end{bmatrix}.
\]
This shows that the undirected join characterization does not extend verbatim to digraphs. For simple digraphs, the corresponding characterization is instead captured by the reaches of the complement.


\begin{proposition}\label{prop:n-eigenvalue-coreaches}
Let $D$ be a simple digraph of order $n$ and suppose $\overline D$ has $k\geq 1$ co-reaches.
Then the algebraic and geometric multiplicity of the eigenvalue $n$ of $\LLL(D)$ is equal to $k-1$.
\end{proposition}

\begin{proof}
By \Cref{lem:comp-poly-variants}, if $n$ has algebraic multiplicity $r$ as an eigenvalue of $\LLL(D)$, then $0$ has algebraic multiplicity $r+1$ as an eigenvalue of $\LLL(\overline D)$. By \Cref{cor:zero-mult-coreaches}, the algebraic (and geometric) multiplicity of the eigenvalue $0$ of $\LLL(\overline D)$ equals the number of co-reaches of $\overline D$, namely $k$. Hence $r+1=k$, so the algebraic multiplicity of $n$ for $\LLL(D)$ is $k-1$.

For the geometric multiplicity, note that $\mathbf{1}\in\ker \LLL(\overline D)$, and by \Cref{cor:zero-mult-coreaches} the nullspace of $\LLL(\overline D)$ has dimension $k$. The complement relation $\LLL(\overline D)=nI-J-\LLL(D)$ implies that a vector $v$ satisfies $\LLL(D)v=nv$ if and only if $\LLL(\overline D)v=0$ and $v\perp \mathbf{1}$. Thus the eigenspace of $n$ for $\LLL(D)$ is exactly
$\ker \LLL(\overline D)\cap \mathbf{1}^\perp$,
which has dimension $k-1$. 
\end{proof}

If $D$ is a join, then $\overline{D}$ has at least two sink SCCs.
\Cref{prop:n-eigenvalue-coreaches} shows this property characterizes the digraphs $D$ with $n\in\spec(\LLL(D))$.

\begin{corollary}\label{cor:n-eigenvalue-coreaches}
Let $D$ be a simple digraph of order $n$. 
Then $n\in\spec(\LLL(D))$
if and only if $\overline D$ has at least two co-reaches, equivalently $\overline D$ has at least two sink SCCs.
\end{corollary}

\Cref{cor:n-eigenvalue-coreaches} and \Cref{thm:Lout-bound} immediately give the following characterization regarding the spectral radius of $\LLL$.

\begin{corollary}
Let $D$ be a simple digraph of order $n$
and $\rho(\LLL(D))$ denote the spectral radius of $\LLL(D)$.
Then $\rho(\LLL(D))\leq n$, with equality if and only if $\overline D$ has at least two co-reaches, equivalently $\overline D$ has at least two sink SCCs.    
\end{corollary}

\section{Bipartite digraphs}\label{sec:bipartite}
For bipartite digraphs, the usual sign-change similarity relates Laplacian and signless-Laplacian variants.
Thus known spectral results that hold for the signless-Laplacian of a digraph also hold for the out-degree Laplacian of a bipartite digraph.
We record the forms needed below.
 
\begin{definition}
Let $D$ be a digraph. 
\begin{itemize}
\item The \defin{out-degree signless Laplacian} and \defin{in-degree signless Laplacian} are
\[
Q_{\mathrm{out}}(D)\coloneqq D_{\mathrm{out}}(D)+A(D),\qquad
Q_{\mathrm{in}}(D)\coloneqq D_{\mathrm{in}}(D)+A(D).
\]
\item The \defin{incidence signless Laplacian} is
\[
Q_{\mathrm{inc}}(D)\coloneqq
D_{\mathrm{out}}(D)+D_{\mathrm{in}}(D)+A(D)+A(D)^{\mathsf T}.
\]
\item 
The \defin{symmetrized out-degree signless Laplacian} and
\defin{symmetrized in-degree signless Laplacian} are
\[
\mathcal S_{\mathrm{out}}^{+}(D)\coloneqq D_{\mathrm{out}}(D)+\tfrac12(A+A^{\mathsf T}),\qquad
\mathcal S_{\mathrm{in}}^{+}(D)\coloneqq D_{\mathrm{in}}(D)+\tfrac12(A+A^{\mathsf T}).
\]
\end{itemize}
\end{definition}

\begin{lemma}\label{lem:bipartite-all}
Let $D$ be a bipartite digraph with vertex partition $V=X\sqcup Y$,
$|X|=n_1$, $|Y|=n_2$, and order the vertices so that $X$ comes first.
Let
\[
U\coloneqq \begin{bmatrix} I_{n_1} & 0\\ 0 & -I_{n_2}\end{bmatrix}.
\]
Then the following hold:
\begin{enumerate}
\item $Q_{\mathrm{out}}(D)=U\,\mathcal L_{\mathrm{out}}(D)\,U$ and
$Q_{\mathrm{in}}(D)=U\,\mathcal L_{\mathrm{in}}(D)\,U$.
\item $\mathcal S_{\mathrm{out}}^{+}(D)=U\,\mathcal S_{\mathrm{out}}(D)\,U$ and $\mathcal S_{\mathrm{in}}^{+}(D)=U\,\mathcal S_{\mathrm{in}}(D)\,U$.
\item $Q_{\mathrm{inc}}(D)=U\,\mathcal L_{\mathrm{inc}}(D)\,U$.
\end{enumerate}
In particular, we have
\begin{enumerate}
\item 
$\spec(\mathcal L_{\mathrm{out}}(D))=\spec(Q_{\mathrm{out}}(D))$ and $
\spec(\mathcal L_{\mathrm{in}}(D))=\spec(Q_{\mathrm{in}}(D))$. 
\item 
$\spec(\mathcal S_{\mathrm{out}}(D))=\spec(\mathcal S_{\mathrm{out}}^{+}(D))$ and $
\spec(\mathcal S_{\mathrm{in}}(D))=\spec(\mathcal S_{\mathrm{in}}^{+}(D))$.  
\item 
$\spec(\mathcal L_{\mathrm{inc}}(D))=\spec(Q_{\mathrm{inc}}(D))$.
\end{enumerate}
\end{lemma}

\begin{proof}
Suppose $D$ has adjacency matrix $A$ and let
$D_{\mathrm{out}}=D_{\mathrm{out}}(D)$,
$D_{\mathrm{in}}=D_{\mathrm{in}}(D)$.
With the chosen ordering, the adjacency matrix $A=A(D)$ has the block form
$A=\begin{bsmallmatrix}0&A_1\\A_2&0\end{bsmallmatrix}$,
so $UAU=-A$ and also $U(A+A^{\mathsf T})U=-(A+A^{\mathsf T})$. 
Since $U$ commutes with all diagonal matrices,
it commutes with $D_{\mathrm{out}}(D)$ and $D_{\mathrm{in}}(D)$.
Therefore,
\[
U(D_{\mathrm{out}}-A)U=D_{\mathrm{out}}+A,\qquad
U(D_{\mathrm{in}}-A)U=D_{\mathrm{in}}+A,
\]
\[
U\Bigl(D_{\mathrm{out}}-\tfrac12(A+A^{\mathsf T})\Bigr)U
= D_{\mathrm{out}}+\tfrac12(A+A^{\mathsf T}),
\]
and the formulas for $Q_{\mathrm{out}}$ and $\mathcal S_{\mathrm{out}}^{+}$ hold (and similarly for $Q_{\mathrm{in}}$ and $\mathcal S_{\mathrm{in}}^{+}$).
Finally,
\[
U\bigl(D_{\mathrm{out}}+D_{\mathrm{in}}-(A+A^{\mathsf T})\bigr)U
= D_{\mathrm{out}}+D_{\mathrm{in}}+(A+A^{\mathsf T})
\]
giving the formula for $Q_{\mathrm{inc}}$.
Similar matrices have the same spectrum.
\end{proof}


\begin{corollary}\label{cor:pf-mono-all}
Let $D$ be a strongly connected bipartite digraph and $M(D)$ be one of the following matrices:
\[
Q_{\mathrm{out}}(D),\ Q_{\mathrm{in}}(D),\ \mathcal S_{\mathrm{out}}^{+}(D),\ \mathcal S_{\mathrm{in}}^{+}(D),\ Q_{\mathrm{inc}}(D).
\]
Then $M$ is a nonnegative irreducible matrix and 
its spectral radius $\rho(M)$ is an eigenvalue of algebraic (and geometric)
multiplicity $1$.

Moreover, if $\widehat D$ is a subdigraph of $D$ (same vertex set), then for each of the
above signless variants $M$ we have $0\le M(\widehat D)\le M(D)$ entrywise, and hence
\[
\rho\bigl(M(\widehat D)\bigr)\ \le\ \rho\bigl(M(D)\bigr).
\]
\end{corollary}

\begin{proof}
All listed matrices are entrywise nonnegative. Since $D$ is strongly connected, $A$ is irreducible; hence $Q_{\mathrm{out}}(D)$ and $Q_{\mathrm{in}}(D)$ are irreducible. Also, the underlying undirected graph is connected, so $A+A^{\mathsf T}$ is irreducible; hence $\mathcal S_{\mathrm{out}}^+(D)$, $\mathcal S_{\mathrm{in}}^+(D)$, and $Q_{\mathrm{inc}}(D)$ are irreducible.

By the Perron--Frobenius theorem, if $M\ge 0$ is irreducible then $\rho(M)$ is an eigenvalue and is simple.
For the monotonicity statement, if $0\le M(\widehat D)\le M(D)$ entrywise, then
$\rho(M(\widehat D))\le \rho(M(D))$ is a standard consequence of Perron--Frobenius theorem.
\end{proof}


When a bipartite digraph is out-regular, we have the following theorem.

\begin{theorem}\label{thm:regular-bipartite-symmetry}
Let $D$ be a bipartite digraph of order $n$ that is $r$-out-regular. Then
\begin{enumerate}
\item $\phi_{\mathcal L_{\mathrm{out}}(D)}(x)=(-1)^n\,\phi_{\mathcal L_{\mathrm{out}}(D)}(2r-x)$.
\item $\phi_{\mathcal S_{\mathrm{out}}(D)}(x)=(-1)^n\,\phi_{\mathcal S_{\mathrm{out}}(D)}(2r-x)$.
\end{enumerate}
In particular, $\spec(\mathcal L_{\mathrm{out}}(D))$ and $\spec(\mathcal S_{\mathrm{out}}(D))$ are symmetric about $r$.
\end{theorem}


\begin{proof}
Define $U$ as in \cref{lem:bipartite-all} and let $A=A(D)$.
Since $D_{\mathrm{out}}(D)=rI$, we have $\mathcal L_{\mathrm{out}}(D)=rI-A$, and 
$Q_{\mathrm{out}}(D)=rI+A=2rI-\mathcal L_{\mathrm{out}}(D)$.
Thus, negating both sides, adding $xI$ and applying \cref{lem:bipartite-all} gives
\[
U\left(xI-\LLL(D)\right)U=(x-2r)I+\LLL(D).
\]
The first equation now follows by taking determinants and noting $\det(U^2)=1$.
The same argument applies to
\[
\mathcal S_{\mathrm{out}}(D)=rI-\tfrac12(A+A^{\mathsf T}),\qquad
\mathcal S_{\mathrm{out}}^{+}(D)=rI+\tfrac12(A+A^{\mathsf T})=2rI-\mathcal S_{\mathrm{out}}(D)
\]
giving the second equation.
\end{proof}

\section{Further Research}
In this paper, we study several Laplacian matrices associated with loopless digraphs, including the incidence, out-degree, in-degree, symmetrized, and skew-symmetrized Laplacians. Several interesting open problems remain. For instance, it would be interesting to obtain sharp upper bounds for the real and imaginary parts of the eigenvalues; related questions have been investigated in \cite{laplacian2}. Another natural direction is to characterize all digraphs with few distinct eigenvalues.

\printbibliography

\end{document}